\documentclass[a4paper,12pt]{article}
\pdfoutput=1
\usepackage[ansinew]{inputenc}
\usepackage{amsfonts}
\usepackage{latexsym}
\usepackage{amsmath}
\usepackage{amssymb}
\usepackage{graphicx}
\usepackage{mathrsfs}
\usepackage{hyperref}
\usepackage{color}
\usepackage{tikz}
\usepackage{geometry}
\usepackage{pstricks}
 \usepackage{xcolor}
\usepackage{tikz}
\usepackage{graphicx}
\usepackage{float}
\usepackage{eepic}

\hypersetup{
	colorlinks,
	linkcolor=blue,
	citecolor=blue
}

\newcommand{\R}{\mathbb{R}}

\newcommand{\C}{\mathbb{C}}
\newcommand{\N}{\mathbb{N}}

\newtheorem{defin}{Definition}[section]

\newtheorem{theorem}[defin]{Theorem}

\newtheorem{exa}[defin]{Example}
\newenvironment{example}{\begin{exa}\rm}{\end{exa}}

\newtheorem{corollary}[defin]{Corollary}
\newenvironment{proof}
{\noindent{\it Proof.}}{\hfill $\Box$\par\vspace{2.5mm}}

\newtheorem{que}{Question}

\newtheorem{pro}{Problem}

\numberwithin{equation}{section}
\title{\bf\Large The growth of transcendental entire solutions of linear difference equations with polynomial coefficients}
\author{Xiong-Feng Liu\footnote{\ Liu is the corresponding author.}, Zhi-Tao Wen\footnote{\ Wen is supported  by the National Natural Science Foundation of China (No.~12471076) and LKSF STU-GTIIT Joint-research Grant (No. 2024LKSFG06).}~~and Can-Xin Zhu}
\date{}
\begin{document}
	
	\maketitle
	\begin{abstract}
In this paper, we study the growth of transcendental entire solutions of linear difference equations
	\begin{equation}
    P_m(z)\Delta^mf(z)+\cdots+P_1(z)\Delta f(z)+P_0(z)f(z)=0,\tag{+}
    \end{equation}
where $P_j(z)$ are polynomials for $j=0,\ldots,m$. At first, we reveal type of binomial series in terms of its coefficients. Second, we give a list of all possible orders, which are less than 1, and types of transcendental entire solutions of linear difference equations $(+)$.
In particular, we give so far the best precise growth estimate of transcendental entire solutions of order less than 1 of $(+)$, which improves results in \cite{CF2009,CF2016}, \cite{I&Y 2004}, \cite{I&W 2025}.
Third, for any given rational number $\rho\in(0,1)$ and real number $\sigma\in(0,\infty)$, we can construct
a linear difference equation with polynomial coefficients which has a transcendental entire solution of order $\rho$ and type $\sigma$. At last, some examples are illustrated for our main theorem.
		

	\end{abstract}

\section{Introduction}
Let us denote the difference operator by $\Delta f(z)=f(z+1)-f(z)$ for a given function $f$, and let $n$ be a nonnegative integer. Define $\Delta^n f(z)=\Delta(\Delta^{n-1} f(z))$ for $n\geq 1$, and write $\Delta^0 f=f$.
It is well known that the linear difference equation of order $m$
    $$
    a_m(z)\Delta^mf(z)+\cdots+a_1(z)\Delta f(z)+a_0(z)f(z)=0
    $$
with entire coefficients $a_j$, $j=0,1,\ldots,m$, has a system of $m$ meromorphic solutions which are linearly independent over the field of periodic functions with period one, see \cite[Theorem~1]{Praagman}. Hence, it can be studied by using the methods of complex analysis, and in particular those of theory of entire functions or Nevanlinna's theory of value distribution of meromorphic functions.

We now start with \cite[Pages 8--9]{Boas1954} or \cite[Pages~1--4]{Levin1964} to recall the study of the growth of entire function $f(z)$. The {\it order of growth} $\rho=\rho(f)$ is defined by
\begin{equation*}
\rho=\limsup_{r\to\infty}\frac{\log\log M(r,f)}{\log r},\quad 0\leq \rho\leq \infty.
\end{equation*}
It is known that when the order $\rho$ is finite, it yields that for every $\varepsilon>0$
\begin{equation*}
M(r,f)=O(e^{r^{\rho+\varepsilon}}),\quad\text{as $r\to\infty$},
\end{equation*}
where $M(r,f)=\max_{|z|\leq r}|f(z)|$. When $0<\rho<\infty$, the type $\tau=\tau(f)$ is defined by
\begin{equation*}
\tau=\limsup_{r\to\infty}\frac{\log M(r,f)}{r^{\rho}},\quad 0\leq \tau\leq \infty.
\end{equation*}
In particular, a positive finite order entire function $f$ is said to be of \emph{minimum type} if $\tau(f)=0$, to be of \emph{mean type} if $\tau(f)\in(0,\infty)$, and to be of \emph{maximum type} if $\tau(f)=\infty$.

In~\cite[Theorem~1.1]{I&Y 2004} Ishizaki and Yanagihara showed that the order of growth of entire solutions $f$ of order $\rho(f)<1/2$ of the linear difference equation
    \begin{equation}\label{linear.eq}
    P_m(z)\Delta^mf(z)+\cdots+P_1(z)\Delta f(z)+P_0(z)f(z)=0
    \end{equation}
has a connection with Newton polygon, where $P_j$, $j=0,1,\dots, m$, $P_m(z)\not\equiv0$, are polynomials with degree $d_j$ for $j=0, 1, \dots, m$. In particular,
\begin{equation}\label{M.eq}
\log M(r,f) = Lr^{\rho(f)}(1+o(1)),
\end{equation}
where $L>0$. Chiang and Feng extended the condition on the order of growth to entire solutions of order strictly less than 1 by the different method, see \cite[Theorem~7.3]{CF2009} and \cite[Theorem~4]{CF2016}.

Ishizaki and the second author in \cite[Theorem~2.2]{I&W 2025} discussed whether
the possible orders given by the Newton polygon for the entire solutions of order less than 1 of the linear difference equation \eqref{linear.eq} could be attained or not, and constructed
all entire solutions of \eqref{linear.eq} by using binomial series and periodic functions. Moreover, Ishizaki and the second author gave the condition on existence of transcendental entire solutions of order less than 1 of difference equations \eqref{linear.eq} and showed a list of all possible orders $\rho(f)$ in \eqref{M.eq}.

The discussion above inspires us to go further study on the more precise growth estimate of transcendental entire solutions of order less than 1 of \eqref{linear.eq}. Our purpose of this paper is to give a list of the growth estimate of transcendental entire solutions of linear difference equations with polynomial coefficients, which includes the type of entire solutions. That is to determine the constant $L$ in \eqref{M.eq}.

In order to achieve our goal, we show the relation between the type of binomial series and its coefficients. It gives us a list of all possible orders, which are less than 1, and types of transcendental entire solutions of linear difference equations \eqref{linear.eq}. We show so far the best precise growth estimate of transcendental entire solutions of order less than 1 of \eqref{linear.eq}, which improves results in \cite{CF2009,CF2016}, \cite{I&Y 2004}, \cite{I&W 2025}. Further, we can construct
a linear difference equation with polynomial coefficients which has a transcendental entire solution of order $\rho$ and type $\sigma$ for any given rational number $\lambda\in(0,1)$ and real number $\sigma\in(0,\infty)$. At the end, some examples are illustrated for our main theorem.

\section{Main results}
Consider difference equations \eqref{linear.eq}. For $j=0,1,\ldots,m$,
we set $d_j=\deg P_j$, and
    \begin{equation}\label{P.eq}
    P_j(z)=A_{j,d_j}z^{d_j}+\cdots+A_{j,1}z+A_{j,0},
    \end{equation}
where $A_{j,i}\in\C$ for $0\leq i\leq d_j$, and $A_{j,d_j}\neq 0$.
We define a strictly decreasing finite sequence of non-negative integers
    \begin{equation}\label{sequence.eq}
    s_1>s_2>\cdots>s_p\geq 0
    \end{equation}
in the following manner. We choose $s_1$ to be the unique integer satisfying
    \begin{equation}\label{s1.eq}
    d_{s_1}=\max_{0\leq k\leq m} {d_k} \quad\text{and}\quad d_{s_1}>d_{k}\quad\text{for all}~0\leq k<s_1.
    \end{equation}
Then given $s_j$, $j\geq 1$, we define $s_{j+1}$ to be the unique integer satisfying
    \begin{equation}\label{sj.eq}
    d_{s_{j+1}}-s_{j+1}> d_{s_{j}}-s_{j}
    \end{equation}
and
    \begin{equation}\label{sj2.eq}
    d_{s_{j+1}}=\max_{0\leq k<s_j} {d_k} \quad\text{and}\quad d_{s_{j+1}}>d_{k}\quad\text{for all}~0\leq k<s_{j+1}.
    \end{equation}
For a certain $p$, the integer $s_p$ will exist, but the integer $s_{p+1}$ will not exist, and the sequence $s_1,s_2,\ldots,s_p$ terminates with $s_p$.
Obviously, $1\leq p\leq m$ and \eqref{sequence.eq} holds.

We mention that the integers $s_1,\ldots,s_p$ in \eqref{sequence.eq} could also be expressed in the following manner:
    $$
    s_1=\min\left\{j:~d_j=\max_{0\leq k\leq m}{d_k}\right\}
    $$
and given $s_{j}$ and $j\geq 1$, we define
    $$
    s_{j+1}=\min\left\{i:~ d_{i}-i> d_{s_{j}}-s_{j}\quad\text{and}\quad d_i=\max_{0\leq k<s_j}{d_k} \right\}.
    $$

From the definition of the sequence \eqref{sequence.eq},
we see that
    $
    d_{s_1}>d_{s_2}>\cdots>d_{s_p}.
    $
Moreover, it is obvious that
    $$
    d_{s_p}-s_p>\cdots>d_{s_2}-s_2>d_{s_1}-s_1.
    $$
Correspondingly, when $p\geq 2$, we define $j=1,2,\ldots,p-1$
    \begin{equation}\label{orderlist.eq}
    \rho_j=1+\frac{d_{s_{j+1}}-d_{s_{j}}}{s_{j}-s_{j+1}},\quad
    \end{equation}
and
    \begin{equation}\label{typelist.eq}
    L_j=\frac{1}{\rho_j}\left|\frac{A_{s_{j+1},d_{j+1}}}{A_{s_j,d_{s_j}}}\right|^{\frac{\rho_j}{(d_{s_{j+1}} - s_{j+1}) - (d_{s_j} - s_j)}}.
    \end{equation}
From \eqref{sequence.eq} to \eqref{typelist.eq}, we observe that $\rho_j$ is rational satisfying $0<\rho_j<1$ and $0<L_j<\infty$ for each $j$, $1\leq j\leq p-1$. Moreover, we see that
    $$
    1>\rho_1>\rho_2>\cdots>\rho_p>0.
    $$

By using the nations above, let us state our main result as follows, which improves the results in \cite{CF2009,CF2016}, \cite{I&Y 2004}, \cite{I&W 2025}.

\begin{theorem}\label{Thm.2}
Suppose that $p\geq 2$. We have the following statements.
\begin{itemize}
\item
For any given real pairs $(\rho_j, L_j)$ in \eqref{orderlist.eq} and \eqref{typelist.eq}, there exists at least one transcendental entire solution $f$ of \eqref{linear.eq} such that $\rho(f)=\rho_j$ and $\tau(f)=L_j$.
\item
If $f$ is a transcendental entire solution of order $\rho(f)<1$ of \eqref{linear.eq}, then
    $$
    \log M(r,f)=L_jr^{\rho_j}(1+o(1))
    $$
 for some $j=1,\ldots,p-1$, where $\rho_j$ and $L_j$ are defined in \eqref{orderlist.eq} and \eqref{typelist.eq}.
 \end{itemize}
\end{theorem}

According to Theorem \ref{Thm.2}, we have the following corollary.

\begin{corollary}
If $f$ is a transcendental entire solution of order $\rho(f)<1$ of \eqref{linear.eq}, then
$f$ is of mean type.
\end{corollary}

Theorem \ref{Thm.2} shows that if $f$ is an transcendental entire solution of order of growth $\rho(f)<1$ of linear difference equations with polynomial coefficients, then $f$ is of rational order and mean type. The next theorem reveals that we can construct a linear difference equation with polynomial coefficients, which has an entire solution of any rational order between 0 and 1 and any mean type.

\begin{theorem}\label{construct.theorem}
For any positive rational number $\lambda\in(0,1)$ and real number $\sigma\in(0,\infty)$, there exists
a linear difference equation of the form \eqref{linear.eq} with polynomial coefficients such that $f$ is an entire solution of \eqref{linear.eq}, whose order is $\rho(f)=\lambda$ and type is $\tau(f)=\sigma$.
\end{theorem}

\section{Binomial series}

We recall and study the properties of binomial series,~\cite{I&Y 2004},~\cite{I&W 2021}.
Define $z^{\underline{0}}=1$ and
\begin{equation*}
z^{\underline{n}}=z(z-1)\cdots(z-n+1)=n!\begin{pmatrix}
z\\
n
\end{pmatrix},\quad n=1, 2, 3, \dots,
\end{equation*}
which is called a {\it falling factorial}.
This yields
\begin{equation*}
\Delta z^{\underline{n}}=(z+1)^{\underline{n}}-z^{\underline{n}}=nz^{\underline{n-1}}\, ,
\end{equation*}
which corresponds to $(z^n)'=nz^{n-1}$ in the differential calculus. Consider the formal series of the form
\begin{equation}\label{3.1}
f(z)=\sum_{n=0}^\infty a_n z^{\underline{n}},\quad a_n\in\mathbb C,\quad n=0,1,2, \dots.
\end{equation}
Let $\{\alpha_n\}$ be a sequence satisfying $|\alpha_n|\to 0$. We define a quantity concerning  $\{\alpha_n\}$ as
\begin{equation}
\chi(\{\alpha_n\})=\limsup_{n\to\infty}\frac{n\log n}{-\log|\alpha_n|}.\label{3.2}
\end{equation}
It was shown in \cite[Theorem 1.1]{I&W 2021} that if $\chi(\{a_n\})<1$, then the binomial series in \eqref{3.1} converges uniformly on every compact subset in $\mathbb C$. In addition,
 the order of growth of $f$ is $\rho(f)=\chi(\{a_n\})$.

Lindel\"of--Pringsheim theorem reveals the relation between the order of growth (or the type) and the coefficients of power series, see \cite[Pages 9--11]{Boas1954}, \cite[Chapter III]{Lindelof1902}, \cite{Lindelof1903}, \cite[Pages 260-263]{Pringsheim1904}. As for binomial series \eqref{3.1}, the sequence $\{a_n\}$ determines the function completely, it is possible for us to discover all the properties of the function by examining its coefficients.
It can be said that \cite[Theorem 1.1]{I&W 2021} is the analogue of Lindel\"of--Pringsheim theorem for binomial series on the order of growth. It is natural for us to consider the type in terms of coefficients of binomial series. Now let us state our result as follows.

\begin{theorem}\label{Thm.1}
Let $f$ be written as binomial series in \eqref{3.1} such that the sequence of coefficients $\{a_n\}$ satisfies $|a_n|\to 0$ and $\chi(\{a_n\})<1$ in \eqref{3.2}. Then $f$ converges in the complex plane. Moreover, if $f$ is of positive order $\rho=\rho(f)>0$, then $f$ is of type $\tau=\tau(f)$ satisfying
    \begin{equation}\label{Thm.1_E.1}
    \tau=\frac{1}{e\rho}\limsup_{n\rightarrow\infty}\left(n|a_{n}|^{\frac{\rho}{n}}\right).
    \end{equation}
\end{theorem}

\begin{proof}
It follows that $f$ is an entire function of order $\rho(f)\in(0,1)$ by \cite[Theorem 1.1]{I&W 2021}.
Set
    $
    v=\limsup_{n\rightarrow\infty}\left(n|a_{n}|^{\frac{\rho}{n}}\right).
    $
We prove $v \geq e \rho \tau $ at first in \eqref{Thm.1_E.1}.
It is easy to see that the inequality holds when $v=\infty$. Now we assume that $0\leq v<\infty$. Thus, for any $\varepsilon > 0$, there exists $N_0>0$ such that
    $$
    |a_n| < \left( \frac{v + \varepsilon}{n} \right)^{\frac{n}{\rho}}
    $$
holds for $n>N_0$. Therefore, it yields that
    \begin{equation}\label{3.4}
    \begin{split}
    M(r,f) &\leq \left(\sum_{n=0}^{N_{0} - 1} + \sum_{n \ge N_{0}} \right) |a_{n}||z^{\underline{n}}|
    \leq c_{1}r^{N_{0} - 1} + \sum_{n=0}^\infty |a_{n}| |z^{\underline{n}}|\\
    &\leq c_{1}r^{N_{0} - 1} + \left(\sum_{n \leq r^k} + \sum_{n > r^k}\right)|a_{n}| |z^{\underline{n}}|\\
    &=  c_{1}r^{N_{0} - 1} + I_{1} + I_{2},
    \end{split}
    \end{equation}
where $c_{1}$ is a positive constant and $r=|z|$, and $k\in\R$ such that $\rho<k<(1+\rho)/2$. In the following, we proceed to estimate $I_1$ and $I_2$. It is known that $(1+x)^n\leq e^{nx}$ holds for $x>-1$ and $n\in\N^+$. Hence,
    \begin{equation}\label{3.5}
    \begin{split}
    I_{1} &= \sum_{n\leq r^k}|a_{n}| |z^{\underline{n}}|\leq \sum_{n\leq r^k}|a_{n}| (r+n)^n
    \leq \sum_{n\leq r^k}|a_{n}| (1+r^{k-1})^nr^n\\
    &\leq \sum_{n\leq r^k}|a_{n}| r^ne^{nr^{k-1}}
    \leq \sum_{n\leq r^k} |a_{n}| r^ne^{r^{2k-1}}
    \leq \sum_{n\leq r^k} \left(\frac{v+\varepsilon}{n}\right)^{\frac{n}{\rho}} r^ne^{r^{2k-1}}\\
    &\leq r^ke^{r^{2k-1}} \exp \left(\frac{(v +\varepsilon)r^{\rho}}{e\rho}\right).
     \end{split}
    \end{equation}
The last inequality is according to the fact that $(a/x)^{x/b}$ does not exceed its maximum, which is $e^{a/be}$, attained for $x=a/e$, where $a,b,c\in\R$. For $I_2$, it follows that
    \begin{equation}\label{3.6}
    \begin{split}
    I_{2}& =\sum_{n> r^k}|a_{n}| |z^{\underline{n}}|
    \leq \sum_{n > r^k} |a_{n}| (r+n-1)^{n}\\
    &\leq \sum_{n > r^k} |a_{n}| (n^{1/k}+n-1)^{n}
    \leq \sum_{n > r^k} |a_{n}| \left(2n^{\frac1k}\right)^{n}\\
    &\leq \sum_{n > r^k} \left(\frac{2(v+\varepsilon)^{\frac{1}{\rho}}}{n^{\frac{1}{\rho} -\frac1k }}\right)^{n} \leq \sum_{n > r^k} \left(\frac{2(v+\varepsilon)^{\frac{1}{\rho}}}{r^{\frac{k}{\rho} - 1}}\right)^{n}=O(1)
     \end{split}
    \end{equation}
holds for large $r$. From \eqref{3.4} to \eqref{3.6}, we deduce that for sufficiently large $r$
    $$
    M(r,f) \leq c_{1}r^{N_{0} - 1} + r^ke^{r^{2k-1}} \exp \left(\frac{(v +\varepsilon)r^{\rho}}{e\rho}\right) +O(1)
    \leq Kr^ke^{r^{2k-1}} \exp \left(\frac{(v +\varepsilon)r^{\rho}}{e\rho}\right),
    $$
where $K$ is a positive constant. Therefore, we have
    $$
    \tau = \limsup_{r \to \infty} \frac{\log M(r,f)}{r^{\rho}} \leq \frac{v+2\varepsilon}{e\rho}.
    $$
Now let $\varepsilon \to 0$. It gives us $v \ge e\rho\tau.$

\medskip			
In the following, we proceed to prove that $v \leq e\rho\tau$. It is easy to see that this inequality holds when $v=0$. Let us assume that $v>0$. For any given \( \varepsilon > 0 \), there exists a sequence $\{n_{j}\}$ satisfying $n_j\to\infty$ as $j\to\infty$ such that
    \begin{equation}\label{Pf_Thm1.1_E.1}
    |a_{n_{j}}| > \left(\frac{v-\varepsilon}{n_{j}}\right)^{\frac{n_{j}}{\rho}}.
    \end{equation}
Moreover, for $r > n$ we have
    \begin{align*}
    a_n &= \frac{1}{2\pi i} \int_{|\zeta|=r} \frac{f(\zeta)}{\zeta^{\underline{n+1}}} d\zeta \\
    &= \frac{1}{2\pi i} \int_{|\zeta|=r} \frac{f(\zeta)}{\zeta(\zeta-1)\cdots(\zeta-n+1)(\zeta-n)} d\zeta,
    \end{align*}
which corresponds to the Cauchy integral formula, see \cite[Remark~2.2]{I&Y 2004}. Hence,
    \begin{equation}\label{an.eq}
    |a_n|r^{n} \le \frac{r^{n}}{(r-1)(r-2)\cdots(r-n)} M(r,f)=\frac{M(r,f)}{\prod_{k=1}^{n}\left(1-\frac{k}{r}\right)}
    \end{equation}
for $r>n$. Using the inequality $\log(1 - x) \geq -x - x^2$  for $x < \frac{1}{2}$, we have
     \begin{align*}
    \sum_{k=1}^n \log\left(1 - \frac{k}{r}\right) &\geq -\sum_{k=1}^n \left(\frac{k}{r} + \frac{k^2}{r^2}\right)
    = -\left( \frac{n(n+1)}{2r} + \frac{1}{r^2} \cdot \frac{n(n+1)(2n+1)}{6}\right)\\
    &
    \geq -\left(\frac{n(n+1)}{r} + \frac{n(n+1)^2}{r^2}\right)
    \end{align*}
holds when $r>n$. Therefore, it yields that
    \begin{align}\label{Pf_Thm1.1_E.2}
    \frac{r^{n}}{(r-1)(r-2)\cdots(r-n)} &\leq \exp\left(\frac{n(n+1)}{r} + \frac{n(n+1)^2}{r^2}\right)
    \end{align}
holds when $r>2n$. Choose a sequence $\{r_j\}$ such that $r_{j}^{\rho} = \frac{n_{j}e}{v -\varepsilon}$. It gives us that $r_j>2n_j$ for $r_j>\left(2(v-\varepsilon)/e\right)^{1/(1-\rho)}$.
 From \eqref{Pf_Thm1.1_E.1} to \eqref{Pf_Thm1.1_E.2}, we conclude that when $r_j>\left(2(v-\varepsilon)/e\right)^{1/(1-\rho)}$, the sequence $\{r_j\}$ satisfies
 $r_j\to\infty$ as $j\to\infty$, and
    \begin{align*}
    \exp \left( \frac{(v- \varepsilon)r_{j}^{\rho}}{e\rho}\right) = \left(\frac{v-\varepsilon}{n_{j}}\right)^{\frac{n_{j}}{\rho}} r_{j}^{n_{j}}
    \leq |a_{n_{j}}|r_{j}^{n_{j}}
    \leq \exp\left(\frac{n_j(n_j+1)}{r_j} + \frac{n_j(n_j+1)^2}{r_j^2}\right)M(r_{j},f).
    \end{align*}
Therefore, we have
    \begin{align*}
    \frac{v-\varepsilon}{e\rho}
\leq \limsup_{j\to\infty}\left(\frac{n_j(n_j+1)}{r_j^{1+\rho}}+ \frac{n_j(n_j+1)^2}{r_j^{2+\rho}}+
\frac{\log M(r_{j},f)}{r_j^{\rho}}\right)
\leq \tau.
    \end{align*}
Let $\varepsilon \to 0$. It yields that $v \leq e\rho\tau$. We prove our assertion.
    \end{proof}
		
It was shown from Theorem~\ref{Thm.1} and \cite[Theorem~1.1]{I&W 2021} that for any real pairs $(\rho,\tau)$, where $0<\rho<1$, we can construct a binomial series of the form \eqref{3.1} of order of growth $\rho$ and type $\tau$. In the following, we give several examples to illustrate this technique.
			
\begin{example}
Let $0<\rho<1$. The binomial series
    $$
    f(z) =\sum_{n=0}^{\infty} \left(\frac{\log^\rho n}{n}\right)^{\frac n\rho} z^{\underline{n}}
    $$
converges in every compact subset in $\C$. In addition, the entire function $f$ is of order $\rho(f)=\rho$ and maximum type $\tau(f) = \infty$.
\end{example}

\begin{example}
Let $0<\rho<1$ and $0<\tau<\infty$. The binomial series
    $$
    f(z) =\sum_{n=0}^{\infty} \left(\frac{e\rho \tau}{n}\right)^{\frac n\rho} z^{\underline{n}}
    $$
converges in every compact subset in $\C$. Moveover, the entire function $f$ is of order $\rho(f)=\rho$ and mean type $\tau(f) = \tau\in(0,\infty)$.
\end{example}
				
\begin{example}
Let $0<\rho<1$. The binomial series
    $$
    f(z) =\sum_{n=0}^{\infty} \frac{1}{\left(n^{\frac1\rho}\log n\right)^{n}}z^{\underline{n}}
    $$
converges in every compact subset in $\C$. The entire function $f$ is of order $\rho(f)=\rho$ and minimum type $\tau(f) = 0$.
\end{example}
		
\section{Proof of Theorem~\ref{Thm.2}}\label{Pf-Thm.2}

Suppose that there exists a formal solution of \eqref{linear.eq} of the following form
    $$
    f(z)=\sum_{n=0}^\infty a_nz^{\underline{n}}.
    $$
Obviously, for any given $j\in\N^+$, we have
    $$
    \Delta^j f(z)=\sum_{n=0}^\infty a_{n+j}(n+j)^{\underline{j}}z^{\underline{n}}.
    $$
For any given $j=0,1,\ldots,m$, it follows from \cite[Corollary A.1]{I&W 2021} or \cite[Page 38]{KP2001} that the polynomial $P_j$ in \eqref{P.eq} is written as
    $$
    P_j(z)=H_{j,d_j}z^{\underline{d_j}}+H_{j,d_j-1}z^{\underline{d_j-1}}
    +\cdots+H_{j,1}z+H_{j,0},
    $$
where $H_{j,i}\in\C$ for $0\leq i\leq d_j$ and $H_{j,d_j}=A_{j,d_j}\neq 0$ in \eqref{P.eq}.
Moreover, it yields from \cite[Lemma~3.1]{I&W 2025} that
    \begin{equation}\label{4.1}
    P_j(z)\Delta^j f(z)=\sum_{n=0}^\infty \sum_{t=0}^{d_j}\sum_{k=0}^t
    a_{n+j}H_{j,t}\binom{t}{k}(n+j)^{\underline{k+j}}z^{\underline{n+t-k}}
    \end{equation}
for $j=1,\ldots,m$. Hence, it yields from \eqref{linear.eq} and \eqref{4.1} that
    \begin{equation}\label{sum.eq}
    \sum_{j=0}^m\sum_{n=0}^\infty \sum_{t=0}^{d_j}\sum_{k=0}^t
    a_{n+j}H_{j,t}\binom{t}{k}(n+j)^{\underline{k+j}}z^{\underline{n+t-k}}=0.
    \end{equation}
We set $i=t-k$ in \eqref{sum.eq} and obtain that
    \begin{equation}\label{4.2}
    \begin{split}
    &\sum_{j=0}^m\sum_{n=0}^\infty a_{n+j} (n+j)^{\underline{j}} \sum_{t=0}^{d_j}\sum_{k=0}^t H_{j,t}\binom{t}{k}n^{\underline{k}}z^{\underline{n+t-k}}\\
    =&\sum_{j=0}^m\sum_{n=0}^\infty a_{n+j}(n+j)^{\underline{j}}
    \sum_{i=0}^{d_j}\frac{\Delta^i(P_j(n))}{i!}z^{\underline{n+i}}=0.
    \end{split}
    \end{equation}
Now, let us assume that $d=\max_{0\leq j\leq m}\{d_j\}$. We reduce that from \eqref{4.2}
    \begin{equation}\label{4.3}
    \begin{split}
    &\sum_{j=0}^m\sum_{n=0}^\infty a_{n+j}(n+j)^{\underline{j}}
    \sum_{i=0}^{d_j}\frac{\Delta^i(P_j(n))}{i!}z^{\underline{n+i}}\\
    =&\left(\sum_{n=0}^{d-1}\sum_{i=0}^n+    \sum_{n=d}^\infty\sum_{i=0}^d\right)\sum_{j=0}^m
     a_{n-i+j}(n-i+j)^{\underline{j}}\frac{\Delta^i(P_j(n-i))}{i!}z^{\underline{n}}=0.
    \end{split}
    \end{equation}
Using the fact that $z^{\underline{n}}$ are linearly independent over the periodic field with period one for distinct $n$, we obtain from \eqref{4.3} that
    \begin{equation}\label{np1.eq}
    \sum_{i=0}^n\sum_{j=0}^m
     a_{n-i+j}(n-i+j)^{\underline{j}}\frac{\Delta^i(P_j(n-i))}{i!}=0
     \quad\text{for}\quad n< d,
     \end{equation}
and
    \begin{equation}\label{np2.eq}
    \sum_{i=0}^d\sum_{j=0}^m
     a_{n-i+j}(n-i+j)^{\underline{j}}\frac{\Delta^i(P_j(n-i))}{i!}=0
     \quad\text{for}\quad n\geq d.
     \end{equation}
We consider the asymptotic behaviour of $a_n$ for large $n$ from \eqref{np2.eq} and ignore \eqref{np1.eq}. We can assume that $\Delta^\alpha f(z)=0$ for $\alpha<0$. We see that $a_n$ satisfies a recurrence relation of $m+d$ order
    \begin{equation}\label{recu.eq}
    a_{n+m}Q(n,-m)+a_{n+m-1}Q(n,-m+1)+\cdots+a_{n-d}Q(n,d)=0,
    \end{equation}
where
    $$
    Q(n,i)=\sum_{j=0}^m\frac{(n-i)^{\underline{j}}}{(i+j)!}\Delta^{i+j}(P_j(n-j-i)).
    $$
Since $p\geq 2$ and the relation \eqref{s1.eq} to \eqref{sj2.eq} hold, we have
    \begin{equation}\label{Qdegree.eq}
    \begin{split}
    \deg{Q(n,k)}&\leq d_{s_{1}}-k~~\quad\text{for}\quad -m\leq k<d_{s_1}-s_{1};\\
    \deg{Q(n,k)}&=s_j\quad\quad\quad\quad\text{for}\quad k=d_{s_j}-s_j;\\
    \deg{Q(n,k)}&\leq d_{s_{j+1}}-k\quad\text{for}\quad d_{s_j}-s_j<k<d_{s_{j+1}}-s_{j+1};\\
    Q(n,k)&=0 ~\quad\quad\quad\quad\text{for}\quad d_{s_p}-s_p<k\leq d.
    \end{split}
    \end{equation}
According to \eqref{Qdegree.eq}, we rewrite \eqref{recu.eq} as
    \begin{equation}\label{recu4.eq}
     a_{n+m}Q(n,-m)+a_{n+m-1}Q(n,-m+1)+\cdots+a_{n-(d_{s_p}-s_p)}Q(n,d_{s_p}-s_p)=0.
    \end{equation}
Let us set
    $$
   a_{n+m}=\frac{x(n)}{[(n+m)!]^\mu},
    $$
where $\mu\in\R$. Then $x(n)$ satisfies the recurrence relation
    \begin{equation*}\label{xrelation.eq}
    x(n)T(n,\mu,0)+x(n-1)T(n,\mu,1)+\cdots+x(n-(m+d_{s_p}-s_p))T(n,\mu,m+d_{s_p}-s_p)=0,
    \end{equation*}
where
    $$
    T(n,\mu,i)=Q(n,i-m)[(n+m)^{\underline{i}}]^\mu
    $$
for $i=0,1,\ldots,m+d_{s_p}-s_p$. The highest power on $n$ of $T(n,\mu,i)$ is denoted by $\deg{T(n,\mu,i)}$. Obviously,
    $$
    \deg{T(n,\mu,i)}=\deg{Q(n,i-m)}+i\mu
    $$
for $i=0,1,\ldots,m+d_{s_p}-s_p$. Since we assume that $p\geq 2$, for given $j=1,2,\ldots,p-1$ and for any $\mu>1$, we deduce from \eqref{Qdegree.eq} that
    \begin{equation}\label{Tdegree.eq}
    \begin{split}
    &\deg{T(n,\mu,m+d_{s_{j+1}}-s_{j+1})}>\deg{T(n,i)}~~\text{for}\quad m+d_{s_j}-s_j<i<m+d_{s_{j+1}}-s_{j+1};\\
    &\deg{T(n,\mu,m+d_{s_1}-s_1})>\deg{T(n,i)} ~\quad\quad\text{for}\quad 0\leq i<m+d_{s_1}-s_{1}.\\
    \end{split}
    \end{equation}
Now let us choose $\mu_j=1/\rho_j$ for $j=1,2,\ldots,p-1$, we have
    \begin{equation}\label{Teq.eq}
    \deg{T(n,\mu_j,m+d_{s_j}-s_j)}=\deg{T(n,\mu_j,m+d_{s_{j+1}}-s_{j+1})}
    \end{equation}
for any $j=1,2,\ldots,p-1$. Moreover, if $p\geq 3$, it yields that for any given $j=1,2,\ldots,p-1$
    \begin{equation}\label{Tbig.eq}
    \begin{split}
    &\deg{T(n,\mu_j,m+d_{s_{j}}-s_{j})}-\deg{T(n,\mu_j,m+d_{s_k}-s_k)}\\
    =&s_{j}+(m+d_{s_{j}}-s_{j})\mu_j-s_k-(m+d_{s_k}-s_k)\mu_j\\
    =&(s_j-s_k)+(d_{s_j}-s_j+s_k-d_{s_k})\mu_j\\
    =&(s_j-s_k)\left(1-\left(\frac{d_{s_j}-d_{s_k}}{s_k-s_j}+1\right)\mu_j\right)\\
    >&(s_j-s_k)\left(1-\left(\frac{d_{s_j}-d_{s_{j+1}}}{s_{j+1}-s_j}+1\right)\mu_j\right)\\
    =&(s_j-s_k)(1-\mu_j/\mu_{j})\\
    =&0
    \end{split}
    \end{equation}
holds for $j+1<k\leq p$. Thus, from \eqref{Tdegree.eq} to
\eqref{Tbig.eq}, we deduce that for a given $j=1,2,\ldots,p-1$,
    \begin{equation}\label{4.13}
     \deg{T(n,\mu_j,m+d_{s_j}-s_j)}=\deg{T(n,\mu_j,m+d_{s_{j+1}}-s_{j+1})}
     >\deg{T(n,\mu_j,i)}
    \end{equation}
holds for $i\neq m+d_{s_j}-s_j$ and $i\neq m+d_{s_{j+1}}-s_{j+1}$.

Let us denote $y_x=\deg Q(n,x-m)$ for $x=m+d_{s_1}-s_1,\ldots,m+d_{s_p}-s_p$. Choose $x-$ and $y-$ axes, plot the points $(x,y_x)$ as in Figure 1. Construct broken lines $L$, convex upward, such that both ends of each segment of the line are points of the set $(x,y_x)$ and such that all points of the set lie upon or beneath the line, see e.g., \cite[P. 511]{Adams 1928}.

    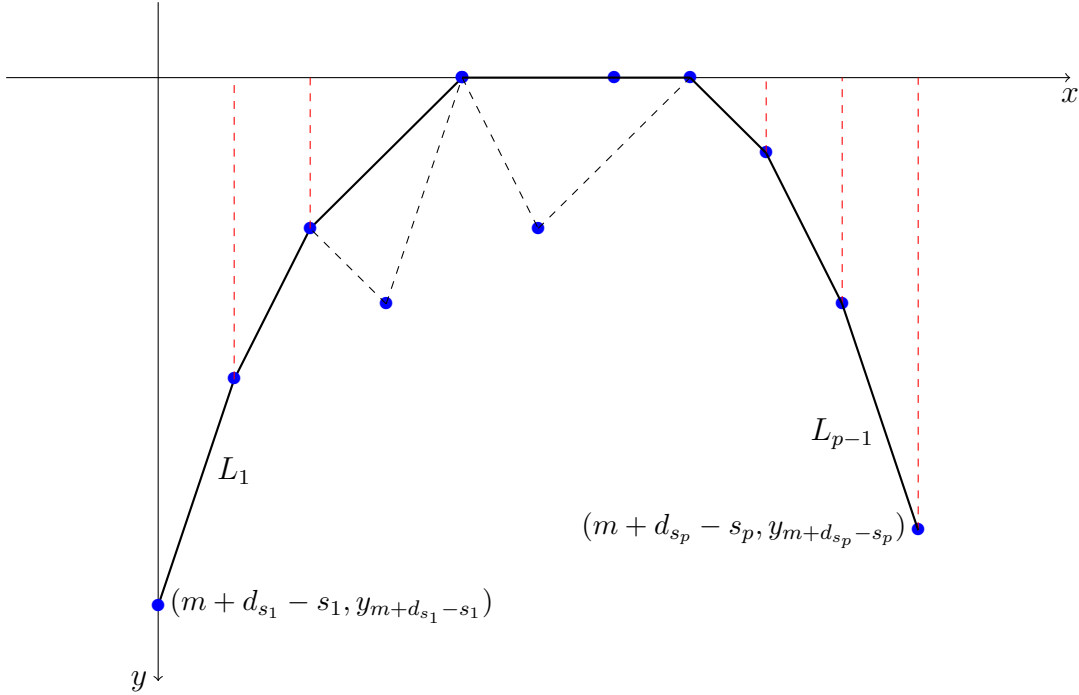
\begin{figure}[H]\label{Line.fi}
    \begin{center}
    \begin{tikzpicture}[scale=1]
    \draw[->](-2,0)--(12,0)node[left,below]{$x$};
    \draw[->](0,1)--(0,-8)node[left]{$y$};
    \draw[.](0,0)--(0,0);
    \draw[-,thick](0,-7)--(1,-4)node[below=25pt]{$L_1$};
    \draw[-,thick](0,-7)--(0,-7)node[below,right]{\small{($m+d_{s_1}-s_1,y_{m+d_{s_1}-s_1}$)}};
    \draw[-,thick](1,-4)--(2,-2);
    \draw[-,thick](2,-2)--(4,0);
    \draw[.](0,-7)node{\blue$\bullet$};
    \draw[.](1,-4)node{\blue{$\bullet$}};
    \draw[.](2,-2)node{\blue{$\bullet$}};
    \draw[.](4,0)node{\blue{$\bullet$}};
    \draw[.](4,0)node{\blue{$\bullet$}};
    \draw[.](3,-3)node{\blue{$\bullet$}};
    \draw[.](5,-2)node{\blue$\bullet$};
    \draw[.](6,0)node{\blue{$\bullet$}};
    \draw[.](7,0)node{\blue{$\bullet$}};
    \draw[.](8,-1)node{\blue{$\bullet$}};
    \draw[.](9,-3)node{\blue{$\bullet$}};
    \draw[.](10,-6)node{\blue{$\bullet$}};
    \draw[draw=red,dashed](1,-4)--(1,0);
    \draw[draw=red,dashed](2,-2)--(2,0);     \draw[draw=red,dashed](8,-1)--(8,0);
    \draw[draw=red,dashed](9,-3)--(9,0);     \draw[draw=red,dashed](10,-6)--(10,0);
    \draw[dashed](3,-3)--(2,-2);
    \draw[dashed](3,-3)--(4,0);
    \draw[dashed](4,0)--(5,-2);
    \draw[dashed](5,-2)--(7,0);
    \draw[-,thick](4,0)--(7,0);
    \draw[-,thick](7,0)--(8,-1);
    \draw[-,thick](8,-1)--(9,-3)node[below=39pt]{$L_{p-1}$};
    \draw[-,thick](9,-3)--(10,-6)node[below,left]{\small{($m+d_{s_p}-s_p,y_{m+d_{s_p}-s_p}$)}};
    \end{tikzpicture}
    \end{center}
	\begin{quote}
    \caption{Convex curve of linear difference equations \eqref{recu4.eq}}
	\end{quote}
    \end{figure}

Obviously, the number of these segments $L$ in Figure 1 is $p-1$ by \eqref{4.13}.
The segments are denoted by $L_1,\ldots,L_{p-1}$, respectively, and slopes of such segments are $\mu_1,\ldots,\mu_{p-1}$ , which are rational numbers.
Due to \cite{Adams 1928} the characteristic equation associated with that segment $L_j$ is
    \begin{equation}\label{ch.eq}
    B_{j+1}\gamma_j^{m+d_{s_j}-s_j}+B_{j}\gamma_j^{m+d_{s_{j+1}}-s_{j+1}}=0,
    \end{equation}
where $B_{j}$ is the coefficient of $n^{s_j}$ in $Q(n,d_{s_j}-s_j)$. In addition, since $m\geq s_j$ and \eqref{s1.eq} to \eqref{sj2.eq}, we have
    \begin{align*}
    Q(n,d_{s_j}-s_j)&=\sum_{k=0}^m\frac{(n-d_{s_j}+s_j)^{\underline{k}}}{(d_{s_j}-s_j+k)!}
    \Delta^{d_{s_j}-s_j+k}(P_k(n-j-d_{s_j}+s_j))\\
    &=\sum_{k=0}^{s_j}\frac{(n-d_{s_j}+s_j)^{\underline{k}}}{(d_{s_j}-s_j+k)!}
    \Delta^{d_{s_j}-s_j+k}(P_k(n-j-d_{s_j}+s_j))\\
    &=A_{s_j,d_{s_j}}n^{s_j}+O(n^{s_j-1}).
    \end{align*}
It implies that $B_j=A_{s_j,d_{s_j}}$ for $j=1,\ldots,p-1$. In addition, there exist $\left(d_{s_{j+1}} - s_{j+1}\right) - \left(d_{s_j} - s_j\right)$ nonzero simple roots of equation \eqref{ch.eq}, which are denoted by $\gamma_{j, t}$ for $t = 1, 2, \ldots, \left(d_{s_{j+1}} - s_{j+1}\right) - \left(d_{s_j} - s_j\right)$. Then for a given $j = 1, 2, \ldots, p - 1$, we find at most $\left(d_{s_{j+1}} - s_{j+1}\right) - \left(d_{s_j} - s_j\right)$ linearly independent solutions of asymptotic behaviour as
    \begin{equation}\label{a.eq}
    a_{n}^{(j, t)} \sim n^{-\mu_j n} e^{\mu_j n} \gamma_{j, t}^n e^{L_j(n)} n^{r_j},
    \end{equation}
where $r_j$ are constants and $L_j(n)$ are polynomials in $n^{1/j}$. Moreover, for any given $j=1,2,\ldots,p-1$, we have
    $$
    \chi(\{a_n^{(j,t)}\})=\limsup_{n\to\infty}\frac{n\log n}{-\log|a_n^{(j,t)}|}=\frac{1}{\mu_j}< 1.
    $$
It implies by \cite[Theorem~1.1]{I&W 2021} that binomial series $\sum a_n^{(j,t)}z^{\underline{n}}$ converges on every compact subset in $\C$.
Therefore, it follows from \eqref{ch.eq}, \eqref{a.eq} and Theorem~\ref{Thm.1} that
    \begin{align*}
    L_{j} &= \frac{1}{e\rho_{j}}\limsup_{n\to\infty} n|a_{n}^{(j,t)}|^{\frac{\rho_{j}}{n}}\\
     &= \frac{1}{e\rho_{j}}\limsup_{n\to\infty} n \left|n^{-\mu_j n} e^{\mu_j n} \gamma_{j, t}^n e^{L_j(n)} n^{r_j}\right|^{\frac{\rho_{j}}{n}}\\
    &= \frac{1}{\rho_j}\left|\frac{A_{s_{j+1},d_{s_{j+1}}}}{A_{s_{j},d_{s_{j}}}} \right|^{\frac{\rho_{j}}{\left(d_{s_{j+1}} - s_{j+1}\right) - \left(d_{s_j} - s_j\right)}}
    \end{align*}
for $j=1,\ldots,p-1$. We can finish our proof on the assertion by \cite[Theorem~7.3]{CF2009}, \cite[Theorem~4]{CF2016}
and \cite[Theorem~2.2]{I&W 2025}.
\hfill $\square$

\section{Proof of Theorem~\ref{construct.theorem}}

For any given positive rational number $0<\lambda<1$ and positive real number $\sigma$, we write
    $
    \lambda=q/p,
    $
where $p$ and $q$ are relatively prime positive integer such that $q<p$.
Set a linear difference equation
    \begin{equation}\label{pq.eq}
    A_pz^{\underline{p}}\Delta^p f(z-p)+\cdots+
    A_1z\Delta f(z-1)-A_0z^{\underline{q}}f(z-q)=0,
    \end{equation}
where $A_j$ are constants for $j=0,\ldots,p$.
Consider a formal solution $f(z)=\sum_{n=0}^\infty a_nz^{\underline{n}}$ of \eqref{pq.eq}. Since
    $$
    z^{\underline{k}}\Delta^m\left(\sum_{n=0}^\infty a_n(z-k)^{\underline{n}}\right)
    =\sum_{n=0}^\infty a_nn^{\underline{m}}z^{\underline{n-m+k}}
    $$
for any positive integer $k$ and $m$. We write \eqref{pq.eq} as
    \begin{equation*}\label{Apq.eq}
    A_p\sum_{n=0}^\infty a_nn^{\underline{p}}
    z^{\underline{n}}+\cdots+
    A_1\sum_{n=0}^\infty a_nnz^{\underline{n}}
    -A_0\sum_{n=q}^\infty a_{n-q}z^{\underline{n}}=0.
    \end{equation*}
Hence, it gives us that $a_1=\cdots=a_{q-1}=0$ and
    $$
    f(n)a_{n}=a_{n-q}
    $$
for $n\geq q$, where $f(n)=(A_pn^{\underline{p}}
+\cdots+A_1n)/A_0$. Hence, we see that $a_{qt-1}=\cdots=a_{qt-q+1}=0$ for $t\in\N$ by setting $n=qt$, and get
    \begin{equation}\label{aqt.eq}
    f(qt)a_{qt}=a_{q(t-1)}.
    \end{equation}
Now let us set $f(qt)=(pt)^{\underline{p}}/\sigma^p$, namely,
we choose $A_0,\ldots,A_p$ such that
    \begin{equation}\label{fpq.eq}
    \frac{A_pn^{\underline{p}}
+\cdots+A_1n}{A_0}=\left(\frac{n}{\lambda}\right)^{\underline{p}}\frac{1}{\sigma^p}.
    \end{equation}
It is easy to see that $A_0/A_p=(\lambda\sigma)^p$ and $A_0/A_1=(-1)^{p-1}\lambda\sigma^p/(p-1)!$ from \eqref{fpq.eq}. Thus, it follows from \eqref{aqt.eq} that
    $$
    a_{qt}=\frac{a_{q(t-1)}}{(pt)^{\underline{p}}}\sigma^p=\frac{a_{q(t-2)}}{(pt)^{\underline{p}}(p(t-1))^{\underline{p}}}\sigma^{2p}
    =\cdots=\frac{a_0}{(pt)!}\sigma^{pt}.
    $$
It implies that $f$ is of the form
    $$
    f(z)=a_0\sum_{t=0}^\infty \frac{\sigma^{pt}}{(pt)!}z^{\underline{qt}}.
    $$
By means of \cite[Theorem~1.1]{I&W 2021} and Theorem~\ref{Thm.1}, the formal solution $f$ converges to an entire function of order $\lambda=q/p$ and type $\tau=\sigma$, since
    \begin{equation*}\label{lambda.eq}
    \chi(\{a_{qt}\})=\limsup_{t\to\infty}\frac{qt\log(qt)}{\log\frac{(pt)!}{\sigma^{pt}}}=\frac{q}{p}=\lambda<1,
    \end{equation*}
and
    $$
    \tau(f)=\frac{p}{qe}\limsup_{t\to\infty}qt\left(\frac{\sigma^{pt}}{(pt)!}\right)^{\frac{1}{pt}}=\sigma.
    $$

We have thus proved that \eqref{pq.eq} possesses an entire solution of order $\lambda$ and type $\sigma$.
We set $z+q$ in place of $z$ in \eqref{pq.eq}, and use a formula
\begin{equation*}\label{summer.eq}
\Delta^m f(z+k)= \sum_{j=0}^m \frac{m!}{j!(m-j)!}(-1)^j\sum_{i=0}^{k+m-j}\frac{(k+m-j)!}{i!(k+m-j-i)!}\Delta^i f(z),
 \end{equation*}
for non-negative integers $m$ and $k$. Then we obtain a difference equation of the form \eqref{linear.eq}.

\section{Examples}
In this section we give several examples which illustrate our theorem.

\begin{example}
In \cite[Example~4.3]{I&W 2021} Ishizaki and the second author showed that there exists a
transcendental solution $f$ of order $\rho(f)=1/2$ of the difference equation
\begin{equation}\label{5.1}
(4z+6)\Delta^2 y(z)+3\Delta y(z)+y(z)=0.
\end{equation}
In particular, the entire solution $f$ is of the form
    $$
    f(z)=\sum_{n=0}^\infty \frac{(-1)^n}{(2n)!}z^{\underline{n}}.
    $$
In \cite{I&W 2025} Ishizaki and the second author showed that there exists at most one linearly independent transcendental entire solution of order $1/2$ of \eqref{5.1}.

In addition, it follows by Theorem \ref{Thm.1} that the entire function $f$ is of type
    $$
    \tau(f)=\frac{2}{e}\limsup_{n\to\infty}n\left(\frac{1}{(2n)!}\right)^{\frac{1}{2n}}=1.
    $$
In what follows, let us estimate the growth of entire solutions of order less than 1 of \eqref{5.1} by using Theorem~\ref{Thm.2}. We know $p=2$, $s_1=2$, $s_2=0$ from \eqref{5.1}. Hence, we have $\rho_1=1/2$ and $L_1=1$. It gives that there exists an entire solution of \eqref{5.1} of order $1/2$ and type $1$. Moreover, if $g$ is an entire solution of order $\rho(g)<1$ of \eqref{5.1}, then we have
    $$
    \log M(r,g)=r^{\frac12}(1+o(1)).
    $$
\end{example}

\begin{example}
In \cite[Remark 6.3]{I&Y 2004} Ishizaki and Yanagihara proved that there exists a transcendental entire solution $f$ of order $\rho(f)=1/3$ of the difference equation
    \begin{equation}\label{one.ex}
    (6z^2+19z+15)\Delta^3 f(z)+(z+3)\Delta^2 f(z)-\Delta f(z)-f(z)=0.
    \end{equation}
In \cite[Examples~7.1]{I&W 2025} Ishizaki and the second author proved that there exists at most one linearly independent transcendental entire solution of \eqref{one.ex} of order $1/3$.
Moreover, it was also shown that the entire solution is of the form
    $$
    f(z) = \sum _{n=0}^{ \infty } \frac {1}{6^n \Gamma (n+1) \Gamma (n-1/2) \Gamma (n-1/3)}.
    $$
Before we start to calculate the type of $f$, let us refer the asymptotic formula for the
Gamma function
    $$
    \Gamma(x)=x^{x-1/2}e^{-x}(2\pi)^{1/2}e^{\theta/(12x)},
    $$
where $0<\theta<1$. The result is known as Stirling's theorem, see e.g.,~\cite[Page 58]{Titchmarsh1939},~\cite[Page 253]{WW1927}.
It yields by Theorem \ref{Thm.1} that the entire function $f$ is of type
\begin{align*}
\tau(f) &= \frac{3}{e}\limsup_{n\to\infty} n\left(\frac {1}{6^n \Gamma (n+1) \Gamma (n-1/2) \Gamma (n-1/3)}\right)^{\frac{1}{3n}}\\
&= \frac{3}{e}\limsup_{n\to\infty}
n\left(\frac{e^{3n-\frac16}}{6^{n}(2\pi)^{3/2}(n+1)^{n+1/2}(n-1/2)^{n-1}(n-1/3)^{n-5/6}} \right)^{\frac{1}{3n}}\\
&= \frac{3}{\sqrt[3]{6}}.
\end{align*}
In what follows, let us estimate the growth of entire solutions of order less than 1 of \eqref{one.ex} by using Theorem~\ref{Thm.2}. We know that $p=2$, $s_1 = 3$, $s_2=0$. It yields that $\rho_1=1/3$ and $L_1=3/\sqrt[3]{6}$. Hence, there exists an entire solution of \eqref{one.ex} of order $1/3$ and type $3/\sqrt[3]{6}$. In addition, if $g$ is a transcendental entire solution of \eqref{one.ex} of order $\rho(g)<1$, then
    $$
    \log M(r,g)=\frac{3}{\sqrt[3]{6}}r^{\frac13}(1+o(1)).
    $$
\end{example}
			
\begin{example}
In \cite[Examples~7.2]{I&W 2025} Ishizaki and the second author showed that there exists a transcendental entire solution of order $\rho(f) = 3/4$ of the difference equation
\begin{multline}
(256 z^3+1920z^2+4656z+3640)\Delta^4 y(z)+(384 z^2+1760z+1944)\Delta^3 y(z)\\
-(80z+120)\Delta^2 y(z)-(81 z^2+ 405 z +446)\Delta y(z)-(81z^2+405z+486)y(z)=0.\label{two}
\end{multline}
In particular, the entire solution $f$ is of the form
    $$
    f(z) = \sum_{k=0}^{\infty} \frac{c}{(4k)!}z^{\underline{3k}},
    $$
where $c\in\C$. From Theorem~\ref{Thm.1}, we obtain that the type of the entire solution $f$ is
\begin{align*}
\tau(f) = \frac{4}{3e}\limsup_{k\to\infty} 3k \left(\frac{1}{(4k)!}\right)^{\frac{1}{4k}} = 1.
\end{align*}
In what follows, let us estimate the growth of entire solutions of order less than 1 of \eqref{two} by using Theorem~\ref{Thm.2}. We know that $p=2$, $s_1 = 4$, $s_2 = 0$. Hence, there exists an entire solution of \eqref{two} of order $3/4$ and type $1$.
In addition, if $g$ is an transcendent entire solution of \eqref{two} of order $\rho(g)<1$, then
we have
    $$
    \log M(r,g)=r^{\frac34}(1+o(1)).
    $$
\end{example}

	\bigskip
	\noindent
	\emph{X.-F.~Liu}\\
	\textsc{Shantou University, Department of Mathematics,\\
		Daxue Road No.~243, Shantou 515063, China}\\
	\texttt{e-mail:20xfliu@stu.edu.cn}
	
	\bigskip
	\noindent
	\emph{Z.-T.~Wen}\\
	\textsc{Shantou University, Department of Mathematics,\\
		Daxue Road No.~243, Shantou 515063, China}\\
	\texttt{e-mail:zhtwen@stu.edu.cn}
	
	\bigskip
	\noindent
	\emph{C.-X.~Zhu}\\
	\textsc{Shantou University, Department of Mathematics,\\
		Daxue Road No.~243, Shantou 515063, China}\\
	\texttt{e-mail:18cxzhu@stu.edu.cn}

\end{document}